\documentclass{article}
\usepackage[utf8]{inputenc}

\usepackage{amssymb, amsmath, amsthm, mathtools, wrapfig, graphicx, esvect, xcolor}

\usepackage[breaklinks]{hyperref}
\hypersetup{
	colorlinks = true, 
	urlcolor = cyan, 
	linkcolor = teal, 
	citecolor = cyan 
}

\title{Odd-distance and right-equidistant sets \\ in the maximum and Manhattan metrics}
\author{
  Alexander Golovanov\thanks{MIPT, Moscow, Russia. Email:~\href{mailto:Golovanov@phystech.edu}{\tt Golovanov@phystech.edu}. },
  Andrey Kupavskii\thanks{MIPT, Moscow, Russia and G-SCOP, Universit\'e Grenoble-Alpes, CNRS, France. Research supported by the \href{https://rscf.ru/project/21-71-10092/}{RSF grant N 21-71-10092}. Email:~\href{mailto:kupavskii@ya.ru}{\tt kupavskii@ya.ru}.},
  Arsenii Sagdeev\thanks{MIPT, Moscow, Russia and Alfréd Rényi Institute of Mathematics, Budapest, Hungary. Supported in part by ERC Advanced Grant `GeoScape'. The author is also a winner of Young Russian Mathematics Contest and would like to thank its sponsors and jury. Email:~\href{mailto:sagdeevarsenii@gmail.com}{\tt sagdeevarsenii@gmail.com}.}
}
\date{June 2021}

\newtheorem{theorem}{Theorem}
\newtheorem{lemma}{Lemma}
\newtheorem{claim}{Claim}

\newtheorem{question}{Question}
\theoremstyle{remark}
\newtheorem*{remark}{Remark}

\newcommand{\p}{{\mathcal P}}
\newcommand{\R}{{\mathbb R}}

\newcommand{\N}{{\mathbb N}}

\newcommand{\X}{{\mathbb M}}
\newcommand{\x}{{\mathbf x}}
\newcommand{\y}{{\mathbf y}}
\newcommand{\z}{{\mathbf z}}
\newcommand{\V}{{\boldsymbol v}}
\newcommand{\e}{{\mathbf e}}
\begin{document}
	
\maketitle

\begin{abstract}
	We solve two related extremal-geometric questions in the $n-$dimensio-nal space $\mathbb{R}^n_{\infty}$ equipped with the maximum metric. First, we prove that the maximum size of a {\it right-equidistant} sequence of points in $\mathbb{R}^n_{\infty}$ equals $2^{n+1}-1$. A sequence is {\it right-equidistant} if each of the points is at the same distance from all the succeeding points. Second, we prove that the maximum number of points in $\mathbb{R}^n_{\infty}$ with pairwise odd distances equals $2^n$. We also obtain partial results for both questions in the $n-$dimensional space $\mathbb{R}^n_1$ with the Manhattan distance. 
\end{abstract}

\section{Introduction}

Given a metric space $\X$, its {\it equilateral dimension} $e(\X)$ is the maximum number of its points with pairwise equal distances. It was most extensively studied for the $n-$dimensional $\ell_p-$spaces $\mathbb{R}^n_{p}$. Recall that the $\ell_p-$distance between two points $\x, \y \in \R^n$ is given by
\begin{equation*}
	\|\x-\y\|_p = \big(|x_1-y_1|^p+\ldots+|x_n-y_n|^p\big)^{1/p}
\end{equation*}
for any real $p\ge 1$, and in case $p = \infty$ by
\begin{equation*}
	\|\x-\y\|_{\infty} = \max_i |x_i-y_i|.
\end{equation*}
It is not hard to check (see e.g. \cite{Pet}) that in the Euclidean case we have $e(\R_2^n) = n+1$, while in the max-norm case we have $e(\R_{\infty}^n) = 2^n$. The lower bounds here are given by the vertex sets of a unit simplex and a hypercube, respectively. In contrast, much less is known about the behavior of $e(\R_p^n)$ for $p \neq 2, \infty$. For instance, for the Manhattan distance Alon and Pudl\'ak \cite{AlPud} showed that $e(\R_1^n) < cn\log n$ for some positive constant $c$, while the best lower bound $e(\R_1^n) \ge 2n$ comes from considering the vertices of the standard cross-polytope. Kusner conjectured \cite{Guy} that the lower bound is tight. This conjecture was verified only for $n=3$ (Bandelt, Chepoi, and Laurent \cite{BCL}) and $n=4$ (Koolen, Laurent, and Schrijver \cite{KMS}). For the state of the art for other values of $p$ see \cite{AlPud, Smyth, Swan04, SwanVil}.

In the present paper, we deal with two related problems. The first problem we consider deals with a notion of {\it right-equidistant} sequences. We call a sequence $\x^{(1)},\dots,\x^{(m)}$ of distinct points in $\R_p^n$ {\it right-equidistant,} if $\|\x^{(j_1)}-\x^{(i)}\|_p = \|\x^{(j_2)}-\x^{(i)}\|_p$ for all $1 \le i < j_1 \le j_2 \le m$. Informally,  each point of the sequence is at the same distance from all the succeeding points.

Polyanskii \cite{pol17} proved the following general theorem, improving upon the previous known bound by Nasz\'{o}di, Pach and Swanepoel \cite[Corollary 14]{nps17}.

\begin{theorem} \label{T_nps}
    In any $n$-dimensional normed space the size of a right-equidistant sequence does not exceed $O(3^nn)$.
\end{theorem}

Later, Nasz\'{o}di and Swanepoel \cite{ns17_2} presented an alternative proof of this fact. One motivation for this result is that it implies an upper bound on the cardinality of a set with only $k$ distinct distances between pairs of points (a \emph{$k$-distance} set).

It is not hard to see that the maximum size of a right-equidistant sequence in the Euclidean space $\R_{2}^n$ is equal to $n+2$. Indeed, one can obtain the upper bound by induction on $n$.  As for the lower bound, consider the center of an $n-$dimensional regular simplex along with its $n+1$ vertices\footnote{This extremal configuration is not unique. Actually, there is a continuum of non-isometric extremal configurations. For instance, another natural construction comes from considering a vertex set of an $n-$dimensional regular simplex with one additional point obtained by reflecting its arbitrary vertex along the opposite facet.}. No other partial results are known for non-Euclidean $\ell_p-$spaces $\R_p^n$.

In the present paper, we obtain bounds for the  right-equidistant sequences in the spaces with the maximum metric $\ell_\infty$ and the Manhattan distance $\ell_1$.

\begin{theorem} \label{T_right_infty}
	The maximum size of a right-equidistant sequence of points in~$\R_{\infty}^n$ equals  $2^{n+1}-1$ for all $n \in \N$.
\end{theorem}

\begin{theorem} \label{T_right_1}
	There exists a right-equidistant sequence of $4n-1$ points in $\R_1^n$ for all $n\in\N$.
\end{theorem}

Though we suspect the maximum size of a right-equidistant sequence in $\R_1^n$ to be much closer (if not equal) to the linear lower bound given by Theorem~3 than to the exponential $O(3^nn)$ from Theorem~1, we could not substantially improve upon the last. More specifically, we can only show that any right-equidistant sequence in $\R_1^n$ consists of no more than $3^n$ points, but we will not go into the details considering the insignificance of this improvement\footnote{However, here is the sketch. Take the points of a given right-equidistant sequence one by one and consider the locus of a point that can be added on the next step as a simplicial complex. After the first taken point, this locus is a cross-polytope and thus has $3^n$ faces. Moreover, with each new step the simplicial complex must lose at least one face.}.

The second problem we consider  originates in a paper \cite{GRS} by Graham, Rothschild, and Straus. Given $n \in \N$ and $p \in [1, \infty]$, after a proper scaling, it is easy to see the existence of $e(\R_p^n)$ points in $\R_p^n$ with pairwise {\it unit} distances. In particular, the maximum number of points in $\R_p^n$ with pairwise {\it odd} integral distances is not less than $e(\R_p^n)$. Graham, Rothschild, and Straus \cite{GRS} showed that this trivial lower bound is essentially optimal in the Euclidean case. More precisely, they proved the following.

\begin{theorem}[\cite{GRS}] \label{T_GRS}
	The maximum number of points in $\R_2^n$ with pairwise odd distances equals $n+2$ if $n \equiv 14 \pmod{16}$, and $n+1$ otherwise.
\end{theorem}

Note that the maximum number of points in a normed space $\R^n_N$ with pairwise odd distances can not be bounded from above in general for all $n \ge 2$. (It is easy to see that on the line one can choose at most two points.) Moreover, given the dimension $n \ge 2$, we can construct the norm $N$ such that one can choose arbitrarily many points in $\R^n_N$ with pairwise odd distances. Nevertheless, one can show that this value is bounded in case of the maximum metric via Ramsey theory\footnote{Indeed, let $S$ be an odd-distance set in $\R_{\infty}^n$. Consider the complete graph with $S$ being the set of vertices, and assign to each edge between vertices $\x = (x_1, \ldots, x_n)$ and $\y = (y_1, \ldots, y_n)$ such color $k$ that $\|\x - \y\|_{\infty} = |x_k - y_k|$. Since no two odd integers add up to an odd integer, the graph does not contain a monochromatic triangle. Therefore, $|S| < r(3;n)$, where $r(3;n)$ stands for the multicolor Ramsey number, see \cite{FPS}.}. Moreover, our next result that can be considered as the max-norm analogue of Theorem~4 shows that a natural construction given by the vertices of the unit hypercube is optimal in any dimension.

\begin{theorem} \label{T_odd_infty}
	The maximum number of points in $\R_{\infty}^n$ with pairwise odd distances equals $2^n$ for all $n \in \N$.
\end{theorem}

The situation is much more obscure in case of the Manhattan distance. For all $n \in \N$, we found an explicit configuration of $7n$ points in $\R_1^{3n}$ with pairwise odd distances. This example shows that the vertices of the standard cross-polytope do not provide an optimal construction. On the other hand, some finite upper bound follows from Theorem~\ref{T_odd_infty}. Indeed, since $\R_1^n$ can be isometrically embedded\footnote{For example, by mapping each point $(x_1, \ldots, x_n)$ to the point whose coordinates are all linear combinations of type $x_1\pm x_2\pm\ldots\pm x_n$.} in $\R_{\infty}^{2^{n-1}}$, the size of any odd-distance configuration in $\R_1^n$ does not exceed $2^{2^{n-1}}$.
Our next result provides better upper bound that grows as $n!$ with $n$.
It seems to be an interesting open problem to find the correct asymptotic.

\begin{theorem} \label{T_odd_1}
	The number of points in $\R_{1}^n$ with pairwise odd integral distances does not exceed $n!\cdot n\cdot \ln n \cdot (4+o(1))$ as $n \to \infty$.
\end{theorem}
As a matter of fact, the bound we obtain is actually an upper bound on the {\it chromatic number} of $\R^n_1$ with forbidden odd distances. (We combine it with the trivial fact that the clique number of a graph is at most its chromatic number). Surprisingly, it is not known if the chromatic number of the Euclidean plane with forbidden odd distances is finite or not (see, e.g., \cite{Shelah}).

\section{Odd distance sets}

We split this section into two parts. In the first one, we  deal with the case of the maximum metric and prove Theorem~\ref{T_odd_infty}. In the second one, we consider Manhattan metric and prove Theorem~\ref{T_odd_1}. The proofs of Theorem~\ref{T_odd_infty} and Theorem~\ref{T_right_infty} from Section~\ref{sec_right} are based on the same trick, applied earlier by various authors \cite{BloWil,FKS,Swan99} to other extremal questions about $\R_{\infty}^n$. The trick is to introduce a poset structure on an $n-$dimensional space with the maximum metric.

\subsection{Maximum metric: proof of Theorem~\ref{T_odd_infty}}

First, let us recall the necessary basic notions.

A {\it partially ordered set}, or {\it poset} for shorthand, is a pair $\p = (S,\preceq)$, where $S$ is a set and $\preceq$ is a reflexive, antisymmetric and transitive binary relation on its elements. We call $x,y \in S$ {\it comparable} if $x\preceq y$, and we say that they are {\it incomparable} otherwise. A set of pairwise comparable elements is called a {\it chain}, while a set of pairwise incomparable elements is called an {\it antichain}. The {\it length $\ell(\p)$} and the {\it width $w(\p)$} of the poset $\p$ are the sizes of the largest chain and antichain, respectively. Let us recall Dilworth's theorem.

\begin{theorem}[Dilworth's theorem \cite{Dilworth}] \label{Dilworth's theorem}
	Let $\p = (S,\preceq)$  be an arbitrary finite poset. Then the width $w(\p)$ of $\p$ is equal to the minimum number of disjoint chains that altogether cover $S$. In particular, $|S| \le \ell(\p)w(\p)$.
\end{theorem}

For any $\x = (x_1, \ldots, x_n) \in \R^n$, put $\widehat{\x} \coloneqq (x_1, \ldots, x_{n-1})\in\mathbb{R}^{n-1}$. We define a binary relation $\preceq$ on $\R_{\infty}^n$ by
\begin{equation} \label{eq_poset}
	\x \preceq \y \mbox{ if and only if either } \x=\y \mbox{ or }\|\widehat{\y}-\widehat{\x}\|_\infty < y_n-x_n
\end{equation}
for all $\x, \y \in \R^n$. In particular, if $\x\preceq \y$ and $\x\neq \y$, then $|y_i-x_i|<y_n-x_n$ for all $i \in [n-1]$. Note that we will make use of the strict inequality in the definition. One can easily check that $(\R_{\infty}^n,\preceq)$ is a poset. Indeed, reflexivity and antisymmetry are immediate from the definition, while transitivity follows from the triangle inequality. It is easy to see  that the following  claim holds.

\begin{claim} \label{claim_compare}
	If $\x, \y \in \R^n$ are comparable with respect to \eqref{eq_poset}, then $\|\y-\x\|_\infty = |y_n-x_n|$. If they are incomparable, then $\|\y-\x\|_\infty = \|\widehat{\y}-\widehat{\x}\|_\infty$.
\end{claim}

\begin{proof}[Proof of Theorem~\ref{T_odd_infty}]
	As we mentioned in the introduction, in the light of construction given by the set of vertices of a unit hypercube, we only need to prove the upper bound.
	
	We proceed by induction on $n$. First, observe that  the statement is trivial for $n=1$. Indeed, for any three reals $x<y<z$, all three differences $z-y$, $y-x$, and $z-x = (z-y)+(y-x)$ cannot be simultaneously odd. So, we turn to the induction step.
	
	Fix $n>1$ and let $S$ be a set of points in $\R_{\infty}^n$ with pairwise odd distances. Consider the poset $\p =(S,\preceq)$ with the partial order defined by \eqref{eq_poset}.
	
	First, observe that no three distinct points $\x, \y, \z \in S$ form a chain. Indeed, if $\x \preceq \y \preceq \z$, then by Claim~\ref{claim_compare} all three differences $z_n-y_n$, $y_n-x_n$, and $z_n-x_n$ are odd, which contradicts the base of induction. Thus, $\ell(\p) \le 2$.
	
	Now let  $A\subset S$ be an antichain of size $w(\p)$ with respect to $\preceq$. Put
	$$\widehat{A} = \{\widehat{\x}: \x \in A\} \subset \R_{\infty}^{n-1}.$$
	Claim~\ref{claim_compare} implies that all the pairwise distances between distinct points of $\widehat{A}$ are odd. Therefore,  $w(\p) = |A| = |\widehat{A}| \le 2^{n-1}$ by induction.
	
	Using Dilworth's theorem, we  get that $|S| \le 2\cdot2^{n-1} = 2^n$.
\end{proof}

Observe that this method allows to reduce multidimensional questions on the maximum number of points in $\R_{\infty}^n$ with some additional arithmetical restrictions on the distances between them to one-dimensional number-theoretical problems. For instance, one can effortlessly deduce the following.

\begin{theorem} \label{T_div_infty}
	The maximum number of points in $\R_{\infty}^n$ whose pairwise distances are integers not divisible by $k$ equals  $k^n$ for all $k,n \in \N$.
\end{theorem}

\subsection{Manhattan distance: proof of Theorem~\ref{T_odd_1}}

Let $C= \{\x \in \R^n: \|\x\|_1 < 1/2\}$ be a scaled open $n$-dimensional cross-polytope. It is easy check that the volume $\text{vol}(C)$ of $C$ equals $\frac{1}{n!}$, because the hyperplanes defined by the equations $\{x_i = 0\}$ split $C$ into $2^n$ simplices, each with $n$ orthogonal edges of length $1/2$.

Let $\Lambda$ be a lattice spanned by the vectors $\e_1+\e_n,\dots,\e_{n-1}+\e_n, 2\e_n$, where $\e_i$ stands for the $i$'th standard basis vector. The determinant $\det(\Lambda)$ of this lattice equals $2$ for all $n \in \N$. Besides, the Manhattan distance between any two of its vertices is even.

Therefore, no two points of the disjoint union $\mathcal{C} = \bigsqcup_{\y \in \Lambda}(C+\y)$ are odd integral distance apart. Indeed, assume that $\x^{(1)} \in C+\y^{(1)}$ and $\x^{(2)} \in C+\y^{(2)}$. Recall that $\|\y^{(1)}-\y^{(2)}\|_1$ is even. Put $\|\y^{(1)}-\y^{(2)}\|_1 = 2t$. Then, the triangle inequality implies that
\begin{equation*}
	\|\x^{(1)}-\x^{(2)}\|_1\le \|\y^{(1)}-\y^{(2)}\|_1 +\|\x^{(1)}-\y^{(1)}\|_1 + \|\x^{(2)}-\y^{(2)}\|_1 < 2t+1,
\end{equation*}
and, similarly,
\begin{equation*}
	\|\x^{(1)}-\x^{(2)}\|_1\ge \|\y^{(1)}-\y^{(2)}\|_1 -\|\x^{(1)}-\y^{(1)}\|_1 - \|\x^{(2)}-\y^{(2)}\|_1 > 2t-1.
\end{equation*}

Hence, the maximum number of points in $\R_{1}^n$ with pairwise odd integral distances does not exceed the minimum number of translates of $\mathcal{C}$ that altogether cover $\R^n$. By the classic probabilistic result due to Erd\H{o}s and Rogers \cite{ErRog1962}, the latter value is less than or equal to $\frac{\det(\Lambda)}{\text{vol}(C)} \cdot (2+o(1))n \ln n$ as $n \rightarrow \infty$.

\begin{remark}
Assigning each point with a color corresponding to any of the translates of $\mathcal{C}$ from the abovementioned construction that covers it, we obtain a coloring where the Manhattan distance between no two monochromatic points is an odd integer. Therefore, the construction described above provides an upper bound on the chromatic number of $\R_1^n$ with forbidden odd distances.
\end{remark}

\section{Right-equidistant sequences} \label{sec_right}

We split this section into three parts. In the first part, we prove the upper bound in Theorem~\ref{T_right_infty}. The proof shares some ideas with the previous section, and so we use the same notation. In the last two parts, we prove the lower bound in Theorem~\ref{T_right_infty} and prove Theorem~\ref{T_right_1}, respectively, via explicit constructions.

\subsection{Proof of the upper bound in Theorem~\ref{T_right_infty}}
    Fix a right-equidistant sequence $\x^{(1)}, \ldots, \x^{(m)}$ of points in $\R_{\infty}^n$. We prove by induction on $n$ that $m \le f(n)\coloneqq 2^{n+1}-1$.
    
    First, observe that the case $n=1$ is trivial. Indeed, for any two distinct $x,y \in \R$, there is a unique $z$ not coinciding with $y$, namely $z=2x-y$, such that $|z-x|=|y-x|$, and thus $m \le 3$. So, we turn to the induction step.
    
    Fix $n>1$. Denote by $\p$ the poset $(S, \preceq)$, where $S = \{\x^{(1)}, \ldots, \x^{(m)}\}$ and the partial order $\preceq$ is defined by \eqref{eq_poset}. For any two elements $\y,\z\in S$, we say that $\y$ \textit{occurs earlier than $\z$} if $\y = \x^{(i)}$, $\z = \x^{(j)}$, and $i < j$.

    \begin{lemma} \label{L1}
    $\ell(\p)\leq f(1)=3$.
    \end{lemma}
    
    \begin{proof}
    Assume that elements $\x^{(i_1)},\dots,\x^{(i_k)}$ are pairwise comparable\footnote{Note that we do not state that $\x^{(i_1)}\preceq\dots\preceq\x^{(i_k)}$, since the order on the chain may not coincide with the one in the sequence.}. Clearly, this subsequence is also right-equidistant. By Claim~\ref{claim_compare}, the sequence of their last coordinates is right-equidistant as well. Therefore, $k \le f(1)=3$ from the base of induction. 
    \end{proof}
    
    \begin{lemma} \label{L2}
    $w(\p)\leq f(n - 1)$.
    \end{lemma}
    
    \begin{proof}
    It is clear that any subset $A$ of $S$ is right-equidistant (with a natural order induced from $S$). Moreover, if $A$ is an antichain, then Claim~\ref{claim_compare} implies that
    $$\widehat{A} = \{\widehat{\x}: \x \in A\} \subset \R_{\infty}^{n-1}$$
    is also right-equidistant. Thus, $|A| = |\widehat{A}| \le f(n-1)$ by induction.
    \end{proof}
    
    \begin{lemma} \label{L3}
    Any two chains of size $3$ in $S$ share a common element.
    \end{lemma}
    
    \begin{proof}
    Assume the contrary. Let $\y^{(1)} \preceq \y^{(2)} \preceq \y^{(3)}$ and $\z^{(1)} \preceq \z^{(2)} \preceq \z^{(3)}$ be two chains with all six of their elements being distinct. It follows that $y^{(1)}_n < y^{(2)}_n < y^{(3)}_n$, and $\big\|\y^{(i)} - \y^{(j)}\big\|_{\infty} = y^{(j)}_n - y^{(i)}_n$ for all $1\leq i < j\leq 3$. Similar relations hold for $\z^{(i)}$.

    Note that $\y^{(2)}$ occurs in the sequence earlier than both $\y^{(1)}$ and $\y^{(3)}$. Indeed, $\y^{(1)}$ cannot occur the first from its chain, since $\big\|\y^{(2)} - \y^{(1)}\big\|_{\infty} < \big\|\y^{(3)} - \y^{(1)}\big\|_{\infty}$ by Claim~\ref{claim_compare}. The same holds for $\y^{(3)}$. Similarly, $\z^{(2)}$ occurs earlier than both $\z^{(1)}$ and $\z^{(3)}$.
    
    Without loss of generality\footnote{Indeed, $\y$ and $\z$ are interchangeable, and if the earliest of these points has upper index $3$, then we can replace every point $\x$ of the configuration by $-\x$, thus reverting the $\preceq$ relations in the chains.}, assume that $\y^{(1)}$ occurs earlier than $\y^{(3)}$, $\z^{(1)}$, and $\z^{(3)}$. We claim that for any point $\V \in S$ that occurs later than $\y^{(1)}$ (and, therefore, later than $\y^{(2)}$), the equality $v_n = y^{(3)}_n$ holds.
    
    To show this, denote $d  \coloneqq \big\|\y^{(1)} - \y^{(2)}\big\|_{\infty} = y^{(2)}_n-y^{(1)}_n$. Note that $\big\|\y^{(3)} - \y^{(1)}\big\|_{\infty} = y^{(3)}_n-y^{(1)}_n = 2d$. Hence $\big\|\V - \y^{(1)}\|_{\infty} = 2d$, since $S$ is right-equidistant. Similarly, $\big\|\V - \y^{(2)}\big\|_{\infty} = \big\|\y^{(1)} - \y^{(2)}\big\|_{\infty} = d$. At the same time,
    \begin{equation*}
    	\big\|\widehat{\V} - \widehat{\y}^{(1)}\big\|_{\infty} \leq \big\|\widehat{\V} - \widehat{\y}^{(2)} \big\|_{\infty} + \big\|\widehat{\y}^{(2)} - \widehat{\y}^{(1)}\big\|_{\infty} \leq d + \big\|\widehat{\y}^{(2)} - \widehat{\y}^{(1)}\big\|_{\infty} < 2d,
    \end{equation*}
    where the last strict inequality is from $\y^{(1)} \preceq \y^{(2)}$ by definition\footnote{This is the only place in the present paper where we rely on the fact that the inequality in \eqref{eq_poset} is strict.}. Therefore,
    $$2d = \big\|\V - \y^{(1)}\big\|_{\infty} = \big|v_n-y^{(1)}_n\big|.$$
    
    If $v_n = y^{(1)}_n - 2d$, then $y^{(2)}_n - v_n = 3d$, which contradicts the equality $\left\|\V - \y^{(2)}\right\|_{\infty} = d$. Hence, we get the desired equality
    $$v_n = y^{(1)}_n + 2d = y^{(3)}_n.$$
    
    Applying this to $\z^{(1)}$ and $\z^{(3)}$, we conclude that $z^{(1)}_n = z^{(3)}_n = y^{(3)}_n$, which contradicts the fact that $\z^{(1)} \preceq\z^{(3)}$.
    \end{proof}
    
    Consider any decomposition of $S$ into the smallest possible number of disjoint chains. Dilworth's theorem states that this number equals the width of the poset, which does not exceed $f(n-1)$ by Lemma~\ref{L2}. Besides, Lemmas~\ref{L1} and~\ref{L3} imply that there may be at most one chain of size $3$ in the decomposition, while all other chains should have sizes at most $2$. Thus,
    $$m = |S| \leq 3+2(f(n-1)-1) = f(n),$$
    justifying the  induction step.

\subsection{Proof of the lower bound in Theorem~\ref{T_right_infty}}
    Let $\mathcal{S}$ be the set of all non-empty subsets of $[n]$. More formally, put $$\mathcal{S} = \{0, 1\}^{[n]}\setminus\{\varnothing\}.$$ Fix an ordering $(S_1, S_2, \ldots, S_{2^n-1})$ on $\mathcal{S}$ with the following property: any set occurs in this ordering after all its supersets~--- for example, one can sort these sets in decreasing order of their sizes, or in decreasing order of the numbers they represent when written in binary. For all $i\in[2^n-1]$, let $\V^{(i)}$ be the indicator vector of $S_i$, that is, for all $k\in [n]$, $v^{(i)}_k = 1$  if $k\in S_i$ and $0$ otherwise. 
    
    For all $i\in[2^n-1]$, define $\x^{(2i-1)} \coloneqq 2^{-i}\V^{(i)}$ and $\x^{(2i)} \coloneqq 2^{1-i}\V^{(i)}$. Finally, we put $\x^{(2^{n+1}-1)} \coloneqq \boldsymbol{0} \in \R^n$. We state that the sequence $\left(\x^{(i)}\right)_{i=1}^{2^{n+1}-1}$ is right-equidistant in $\mathbb{R}^n_{\infty}$.
    
    Indeed, let $c_i$ be equal to $\big\|\x^{(i)}\big\|_{\infty}$~--- that is, $0$ for $i = 2^{n+1}-1$, $2^{-(i+1)/2}$ for all other odd $i$, and $2^{1-i/2}$ for all even $i$. We claim that $\big\|\x^{(j)}-\x^{(i)}\big\|_{\infty} = c_i$ for all $1 \le i < j \le 2^{n+1}-1$. To show this, we consider three cases.
    
    First, if $\lceil i/2 \rceil = \lceil j/2 \rceil$, then $\x^{(j)} = 2\x^{(i)}$, and
    $$\big\|\x^{(j)} - \x^{(i)}\big\|_{\infty} = \big\|2\x^{(i)} - \x^{(i)}\big\|_{\infty} = \big\|\x^{(i)}\big\|_{\infty} = c_i.$$
    
    Second, if $\lceil i/2 \rceil < \lceil j/2 \rceil < 2^n$, then
    $$\big\|\x^{(j)} - \x^{(i)}\big\|_{\infty} \leq \max\big\{|z - y|\,\colon\,y\in\{0, c_i\},\,z\in\{0, c_j\}\big\} = \max\{c_i, c_j\} = c_i.$$
    On the other hand, since $S_{\lceil j/2 \rceil}$ is not a superset of $S_{\lceil i/2 \rceil}$, there is a $k\in[n]$ such that $k\in S_{\lceil i/2 \rceil}\!\setminus\! S_{\lceil j/2 \rceil}$. Hence
    $$\big\|\x^{(j)} - \x^{(i)}\big\|_{\infty}\geq \big|x^{(j)}_k - x^{(i)}_k\big| = \big|0-c_i\big| = c_i.$$

    Finally, if $j = 2^{n+1} - 1$, then $\big\|\x^{(i)} - \x^{(j)}\big\|_{\infty} = \big\|\x^{(i)} - \mathbf{0}\big\|_{\infty} = c_i$,
    which completes the proof.

\subsection{Proof of Theorem~\ref{T_right_1}}
	
	Let $(\V^{(1)}, \ldots, \V^{(2n)})$ be the set of all vertices of the standard unit $n-$dimensional cross-polytope translated in such a way that $\V^{(2n)} = {\bf 0}$ with a specific ordering.	Namely, put $\V^{(2i-1)}:=\e_1+\e_{i}$ and $\V^{(2i)}:=\e_1-\e_{n+1-i}$ for all $i \in [n]$, where $\e_i$ stands for the $i$'th standard basis vector.
    
    Then, as in the previous proof, for all $i\in[2n-1]$, define $\x^{(2i-1)} \coloneqq 2^{-i}\V^{(i)}$, and $\x^{(2i)} \coloneqq 2^{1-i}\V^{(i)}$. Finally, put $\x^{(4n - 1)} \coloneqq \V^{(2n)} = \boldsymbol{0}$. We claim that the sequence $\left(\x^{(i)}\right)_{i=1}^{4n-1}$ is right-equidistant in $\mathbb{R}^n_1$.
    
    Indeed, put $c_i:=\big\|\x^{(i)}\big\|_1$. We claim that $\big\|\x^{(j)}- \x^{(i)}\big\|_1 = c_i$ for all $1 \le i < j \leq 4n-1$. To show this, we consider two cases.
    
    First, if $\lceil i/2 \rceil = \lceil j/2 \rceil$, then $\x^{(j)} = 2\x^{(i)}$, and
    $$\big\|\x^{(j)} - \x^{(i)}\big\|_1 = \big\|2\x^{(i)} - \x^{(i)}\big\|_1 = \big\|\x^{(i)}\big\|_1 = c_i.$$
    
    Second, assume that $\lceil i/2 \rceil < \lceil j/2 \rceil$. Note that if the vectors $\x^{(i)}$ and $\x^{(j)}$ have a common nonzero position $k$, $k\in [2,n]$, then $|x^{(i)}_k-x^{(j)}_k| = |x^{(i)}_k|+|x^{(j)}_k|$. The same obviously holds in case one of $x^{(i)}_k,x^{(j)}_k$ is $0$. Using this, we get
    \begin{align*}
    \big\|\x^{(j)} - \x^{(i)}\big\|_1 & = \sum_{k=1}^n\big|x^{(j)}_k - x^{(i)}_k\big| \\
    & = x^{(i)}_1 - x^{(j)}_1 + \sum_{k=2}^n\left(\big|x^{(j)}_k\big| + \big|x^{(i)}_k\big|\right) \\
    & = x^{(i)}_1 - x^{(j)}_1 + x^{(j)}_1 + \sum_{k=2}^n\big|x^{(i)}_k\big| = c_i.
    \end{align*}
In the last two equalities we used that $\x^{(j)}$ has only two nonzero coordinates and that they are equal in the absolute value.  This completes the proof.

\vspace*{5mm}

The similarity between the last two proofs raises the following question.

\begin{question}
Is it true that for any $n-$dimensional normed space $\R_N^n$, the size of any right-equidistant sequence in it is at most $2e(\R_N^n) - 1$?
\end{question}

\subsubsection*{Acknowledgements}

We would like to thank Ilya Bogdanov for pointing out a way to simplify our original proof of Theorem~\ref{T_right_infty}.

\end{document}